\documentclass[12pt,psamsfonts]{amsart}

\usepackage{amsmath, amssymb, amsfonts, amsthm, latexsym}
\usepackage{dsfont}
\usepackage{amscd}
\usepackage[all]{xy}
\usepackage{graphicx, psfrag}

\newcommand{\CC}{\mathbb C}
\newcommand{\FF}{\mathbb F}

\newcommand{\PP}{\mathbb P}
\newcommand{\QQ}{\mathbb Q}

\newcommand{\ZZ}{\mathbb Z}

\DeclareMathOperator{\im}{im}

\DeclareMathOperator{\Proj}{Proj} 
 \DeclareMathOperator{\Spec}{Spec}

\DeclareMathOperator{\Syz}{Syz}

\newtheorem{proposition}{Proposition}[section]
\newtheorem{theorem}[proposition]{Theorem}
\newtheorem{lemma}[proposition]{Lemma}

\theoremstyle{remark}

\newtheorem{definition}[proposition]{Definition}
\newtheorem{remark}[proposition]{Remark}
\newtheorem{example}[proposition]{Example}

\newtheorem{question}[proposition]{Problem}

\newcommand  {\shF}     {\mathcal{F}}

\newcommand  {\shL}     {\mathcal{L}}

\newcommand  {\shS}     {\mathcal{S}}
\newcommand  {\shT}     {\mathcal{T}}
\newcommand  {\lra}     {\longrightarrow}
\newcommand  {\ra}      {\rightarrow}
\newcommand{\comdots}{ , \ldots , }

\newcommand{\plusdots}{ + \ldots + }

\newcommand  {\dual}    {\vee}
\newcommand  {\fom}     {\mathfrak{m}}

\numberwithin{equation}{section}

\renewcommand{\O}{\mathcal{O}}

\newcommand{\repmatrix}{{\alpha}}

\newcommand{\torsor}{{T}}

\newcommand{\solidclosure}{\rm sc}
\newcommand{\rk}{{\operatorname{rk} }}
\newcommand{\forca}{{A}}

\newcommand{\forcazero}{{B}}
\newcommand{\fop}{{\mathfrak p}}

\def\mydate{\number\day\space\ifcase\month \or January\or February\or March\or April\or May\or
June\or July\or August\or September\or October\or November\or
December\fi \space\number\year}

\begin{document}

\sloppy

\title[Forcing algebras, syzygy bundles, and tight closure]
{Forcing algebras, syzygy bundles, and tight closure}

\author[Holger Brenner]{Holger Brenner}
\address{Fachbereich f\"ur Mathematik und Informatik, Universit\"at Osnabr\"uck, Osnabr\"uck, Germany}

\email{hbrenner@uni-osnabrueck.de}


\thanks{ }



\dedicatory{Preliminary version,  \mydate}

\maketitle


\begin{abstract}
We give a survey about some recent work on tight closure and Hilbert-Kunz theory from the viewpoint of vector
bundles. This work is based in understanding tight closure in terms
of forcing algebras and the cohomological dimension of torsors of
syzygy bundles.
These geometric methods allowed to answer some fundamental questions of tight closure,
in particular the equality between tight closure and plus closure in
graded dimension two over a finite field and the rationality of the
Hilbert-Kunz multiplicity in graded dimension two.
Moreover, this approach showed that tight closure may behave weirdly under arithmetic and geometric deformations,
and provided a negative answer to the localization problem.
\end{abstract}


\section*{Introduction}

In this survey article we describe some developments which led to a
detailed geometric understanding of tight closure in dimension two
in terms of vector bundles and torsors. Tight closure is a technique in positive characteristic introduced by M. Hochster and C. Huneke 20 years ago (\cite{hochsterhunekebriancon}, \cite{hunekeapplication}). We recall its definition. Let $R$ be a commutative ring of positive characteristic $p$ with $e$th \emph{Frobenius} homomorphism $F^{e}:R \rightarrow R$, $f \mapsto f^q$, $q=p^{e}$. For an ideal $I$ let $I^{[q]}:=F^{e}(I)$ be the extended ideal under the $e$th Frobenius. Then the \emph{tight closure} of $I$ is given by
$$I^*=\{f \in R:\, \text{there exists } t, \text{not in any minimal prime}, $$
$$\text{ such that } tf^q \in I^{[q]} \text{ for } q \gg 0 \,  \}
$$
(in the domain case this means just $t \neq 0$, and for all $q$). In this paper we will not deal with the applications of tight closure in commutative algebra, homological algebra and algebraic geometry, but with some of its intrinsic problems. One of them is whether tight closure commutes with localization, that is, whether for a multiplicative system $S \subseteq R$ the equality
$$(I^*)R_S =(IR_S)^*$$
holds (the inclusion $\subseteq$ is always true). A directly related question is whether tight closure is the same as plus closure. The \emph{plus closure} of an ideal $I$ in a domain $R$ is defined to be
$$
I^+ \!=\!\{f \in R\! : \text{there exists } R \subseteq S \text{ finite domain extension such that } f \in IS\} .$$
This question is known as the \emph{tantalizing question} of tight closure theory. The inclusion $I^+ \subseteq I^*$ always holds. Since the plus closure commutes with localization, a positive answer to the tantalizing question for a ring and all its localizations implies a positive answer for the localization problem for this ring. The tantalizing question is a problem already in dimension two, the localization problem starts to get interesting in dimension three.

What makes these problems difficult is that there are no exact criteria for tight closure. There exist many important inclusion criteria for tight closure, and in all these cases the criteria also hold for plus closure (in general, with much more difficult proofs). The situation is that the heartlands of ``tight closure country'' and of ``non tight closure country'' have been well exploited, but not much is known about the thin line which separates them. This paper is about approaching this thin line geometrically.

The original definition of tight closure, where one has to check infinitely many conditions, is difficult to apply. The starting point of the work we are going to present here is another description of tight closure due to Hochster as \emph{solid closure} based on the concept of \emph{forcing algebras}. Forcing algebras were introduced by Hochster in \cite{hochstersolid} in  an attempt
to build up a characteristic-free closure operation with similar
properties as tight closure. This approach rests on the fact that $f \in (f_1, \ldots ,f_n)^*$ holds in $R$ if and only if $H^{\dim R}_{\fom} (A) \neq 0$, where $A=R[T_1, \ldots, T_n]/(f_1T_1+ \ldots + f_nT_n -f)$ is the forcing algebra for these data (see Theorem \ref{solidtight} for the exact statement). This gives a new interpretation for tight closure, where, at least at first glance, not infinitely many conditions are involved. This cohomological interpretation can be refined geometrically, and the goal of this paper is to describe how this is done and where it leads to. We give an overview.

We will describe the basic properties of forcing algebras in Section \ref{forcing}. A special feature of the cohomological condition for tight closure is that it depends only on the open subset $D(\fom A) \subseteq \Spec A$. This open subset is a ``torsor'' over $D(\fom) \subseteq \Spec R$, on which the \emph{syzygy bundle} $\Syz(f_1, \ldots , f_n)$ acts. This allows a more geometric view of the situation (Section \ref{torsor}). In general, closure operations for ideals can be expressed with suitable properties of forcing algebras. We mention some examples of this correspondence in Section \ref{closure} and come back to tight closure and solid closure in Section \ref{sectiontightsolid}. 

To obtain a detailed understanding, we restrict in Section \ref{gradedsection} to the situation of a two-dimensional standard-graded normal domain $R$ over an algebraically closed field and homogeneous $R_+$-primary ideals. In this setting, the question about the cohomological dimension is the question whether a torsor coming from forcing data is an affine scheme. Moreover, to answer this question we can look at the corresponding torsor over the smooth projective curve $\Proj R$. This translates the question into a projective situation. In particular, we can then use concepts from algebraic geometry like \emph{semistable bundles} and the \emph{strong Harder-Narasimhan filtration} to prove results. We obtain an exact numerical criterion for tight closure in this setting (Theorems \ref{ssinclusion} and \ref{ssexclusion}). The containment in the plus closure translates to a geometric condition for the torsors on the curve, and in the case where the base field is the algebraic closure of a finite field we obtain the same criterion. This implies that under all these assumptions, tight closure and plus closure coincide (Theorem \ref{tantalizingtheorem}).

With this geometric approach also some problems in Hilbert-Kunz theory could be solved, in particular it was shown that the \emph{Hilbert-Kunz multiplicity} is a rational number in graded dimension two (Theorem \ref{hkformula}). In fact, there is a formula for it in terms of the strong Harder-Narasimhan filtration of the syzygy bundle. In Section \ref{deformation}, we change the setting and look at families of two-dimensional rings parametrized by a one-dimensional base. Typical bases are $\Spec \mathbb Z$ (\emph{arithmetic deformations}) or $\mathbb A^1_K$ (\emph{geometric deformations}). Natural questions are how tight closure, Hilbert-Kunz multiplicity and strong semistability of bundles vary in such a family. Examples of P. Monsky already showed that the Hilbert-Kunz multiplicity behaves ``weirdly'' in the sense that it is not almost constant. It follows from the geometric interpretation that also strong semistability behaves wildly. Further extra work is needed to show that tight closure also behaves wildly under such a deformation. We present the example of Brenner-Katzman in the arithmetic case and of Brenner-Monsky in the geometric case (Examples \ref{brennerkatzmanexample} and \ref{brennermonskyexample}). The latter example shows also that tight closure does not commute with localization and that even in the two-dimensional situation, the tantalizing question has a negative answer, if the base field contains a transcendental element. We close the paper with some open problems (Section \ref{problems}).

As this is a survey article, we usually omit the proofs and refer to the original research papers and to \cite{brennerbarcelona}. I thank Helena Fischbacher-Weitz, Almar Kaid and Axel St\"abler for useful comments.

\section{Forcing algebras}
\label{forcing}

Let $R$ be a commutative ring, let $M$ be a finitely generated $R$-module
and $N \subseteq M$ a finitely generated $R$-submodule.
Let $s \in M$ be an element. The \emph{forcing algebra}
for these data is constructed as follows: choose generators $y_1
\comdots y_m$ for $M$ and generators $x_1 \comdots x_n$ for $N$.
This gives rise to a surjective homomorphism
$\varphi: R^m \rightarrow M $, a submodule $N' = \varphi^{-1}(N)$ and a
morphism $R^n \rightarrow R^m$ which sends $e_i$ to a preimage
$x_i'$ of $x_i$. Altogether we get the commutative diagram with exact rows

$$\xymatrix{ & & R^n &\! \! \!\stackrel{\repmatrix}{ \longrightarrow }\!\! \!  &
R^m &\! \! \! \longrightarrow \!\! \! & M/N & \!\! \! \longrightarrow \! \! \!& \! \! 0 & \! \! \!\!(*) \cr
& & \downarrow  & & \downarrow \varphi & & \downarrow = & &  &\cr
0\! \! &\! \! \!\!  \longrightarrow\! \! \! \! & N & \! \!\!\longrightarrow\!\! \! &M &\!\! \! \longrightarrow\! \! \! &M/N&\!\! \!\longrightarrow\! & 0 & } $$
($\repmatrix$ is a matrix). The element $s $ is represented by $(s_1 \comdots s_m) \in R^m$,
and $s $ belongs to $N$ if and only if the linear equation
$$ \repmatrix \begin{pmatrix} t_1 \cr . \cr . \cr . \cr t_n \end{pmatrix}
  = \begin{pmatrix} s_1 \cr . \cr . \cr . \cr s_m \end{pmatrix}  $$
has a solution. An important insight due to Hochster is that this
equation can be formulated with new variables $T_1 \comdots T_n$,
and then the ``distance of $s$ to $N$'' - in particular,
whether $s$ belongs to a certain closure of $N$ - is reflected by 
properties of the resulting (generic) forcing algebra. Explicitly,
if
$\repmatrix$ is the matrix describing the submodule $N$ as above and if
$(s_1 \comdots s_m)$ represents
$s$, then the forcing algebra is defined by
$$\forca=R[T_1 \comdots T_n]/(\repmatrix T-s) \, ,$$
where $\repmatrix T-s$ stands for
$$\repmatrix \begin{pmatrix} T_1 \cr . \cr . \cr . \cr T_n \end{pmatrix}
  = \begin{pmatrix} s_1 \cr . \cr . \cr . \cr s_m \end{pmatrix}  $$
or, in other words, for the system of \emph{inhomogeneous linear
equations}
$$ \begin{matrix}
 a_{11}T_1 & +& \ldots &+  &  a_{1n}T_n &=& s_1  \cr
a_{21}T_1 & + & \ldots &+ &  a_{2n}T_n &=& s_2  \cr
 & & & & \cr
 a_{m1}T_1 & +& \ldots &+ &  a_{mn}T_n &=& s_m
    \end{matrix}  \,.  $$
In the case of an ideal $I=(f_1 \comdots f_n)$ and $f \in R$ the
forcing algebra is just $\forca= R[T_1 \comdots T_n]/(f_1T_1
\plusdots f_nT_n -f)$. Forcing algebras are given by the easiest algebraic equations at
all, namely linear equations. Yet we will see that forcing algebras
already have a rich geometry. Of course, starting from the data
$(M,N,s)$ we had to make some choices in order to write down a
forcing algebra, hence only properties which are independent of
these choices are interesting.

The following lemma expresses the \emph{universal property} of a forcing
algebra.

\begin{lemma}
Let the situation be as above, and let $R \rightarrow R'$ be a ring homomorphism.
Then there exists an
$R$-algebra homomorphism $\forca \rightarrow R'$ if and only if $
s \otimes 1 \in \im (N \otimes R' \rightarrow M \otimes R')$.
\end{lemma}
\begin{proof}
This follows from the right exactness of tensor products applied to
the sequence $(*)$ above.
\end{proof}

The lemma implies in particular that for two forcing algebras
$\forca$ and $\forca'$ we have (not uniquely determined) $R$-algebra homomorphisms $\forca
\rightarrow \forca'$ and $\forca' \rightarrow \forca$. It also
implies that $s \in N$ if and only if there exists an
$R$-algebra homomorphism
$\forca \rightarrow R$ (equivalently, $\Spec \forca \rightarrow \Spec R$ has a section).

We continue with some easy geometric properties of the mapping
$\Spec \forca \rightarrow \Spec R $.
The formation of forcing algebras commutes with arbitrary base change
$R \rightarrow R'$. Therefore for every point $\fop \in \Spec R$ the
fiber ring $\forca \otimes_R \kappa(\fop)$ is the forcing algebra
given by $$\repmatrix (\fop) T= s(\fop) \, ,$$ which is an
inhomogeneous linear equation over the field $\kappa(\fop)$. Hence
the fiber of $\Spec \forca \rightarrow \Spec R$ over $\fop$ is the
\emph{solution set} to a system of linear inhomogenous equations.

We know from linear algebra that the solution set to such a system
might be empty, or it is an affine space (in the sense of linear
algebra) of dimension $\geq n-m$. Hence one should think of $\Spec
\forca \rightarrow \Spec R $ as a family of affine-linear spaces
varying with the base. Also, from linear algebra we know that such
a solution set is given by adding to one particular solution a
solution of the corresponding system of homogeneous of linear equations.
The solution set to $\repmatrix(\fop)T =0$ is a vector space over
$\kappa(\fop)$, and this solution set is the fiber over $\fop$ of
the forcing algebra of the zero element, namely
$$\forcazero = R[T_1
\comdots T_n]/(\repmatrix T) = R[T_1
\comdots T_n]/( \sum_{i=1}^n a_{ij} T_i,\, j=1 \comdots m)  \, .$$
As just remarked, the fibers of $V= \Spec \forcazero$ over a point
$\fop$
are vector spaces of possibly varying dimensions. Therefore $V$ is in
general not a vector bundle. It is, however, a commutative
\emph{group scheme} over $\Spec R$, where the addition is given by
$$V \times V \longrightarrow V,\, (s_1 \comdots s_n),(s_1' \comdots s_n')
\longmapsto (s_1 +s_1' \comdots s_n + s_n') $$
(written on the level of sections) and the coaddition by
$$ R[T_1
\comdots T_n]/(\repmatrix T)  \rightarrow R[T_1
\comdots T_n]/(\repmatrix T)  \otimes R[\tilde{T}_1
\comdots \tilde{T}_n]/(\repmatrix \tilde{T}) ,\, T_i \mapsto T_i + \tilde{T}_i . $$
This group scheme is the kernel group scheme of the group scheme
homomorphism
$$ \repmatrix: \mathbb A_R^n \longrightarrow \mathbb A_R^m \, $$
between the trivial additive group schemes of rank $n$ and $m$.
We call it the \emph{syzygy group scheme} for the given generators
of $N$.

The syzygy group scheme acts on the spectrum of a forcing algebra
$\Spec \forca$, $\forca = R[T]/(\repmatrix T - s)$ for every $s \in M$.
The action is exactly as in linear algebra, by adding a solution
of the system of homogeneous equations to a solution of the system
of inhomogeneous equations. An understanding of the syzygy group
scheme is necessary before we can understand the forcing algebras.

Although
$V$ is not a vector bundle in general, it is not too far away. Let
$U \subseteq \Spec R$ be the open subset of points $\fop$ where the mapping
$\repmatrix(\fop)$ is surjective. Then the restricted group scheme
$V|_U$ is a vector bundle of rank $n-m$. If $M/N$ has its support in a maximal ideal $\fom$,
then the syzygy group scheme induces a vector bundle on the \emph{punctured
spectrum}
$\Spec R - \{\fom\}$,
which we call the \emph{syzygy bundle}. Hence on $U$ we have a short
exact sequence
$$0 \lra \Syz \lra \O_U^n \lra \O_U^m \lra 0 \, \, \, $$
of vector bundles on $U$.

We will mostly be interested in the situation where the submodule is an ideal $I \subseteq R$ in the ring. We usually fix ideal generators $I=(f_1, \ldots ,f_n)$ and $(*)$ becomes
$$ R^n \stackrel{f_1, \ldots ,f_n}{\longrightarrow} R \longrightarrow R/I \longrightarrow 0 \, .$$
The ideal generators and an element $f \in R$ defines then the forcing equation $f_1T_1+ \ldots f_nT_n-f =0$. Moreover, if the ideal is primary to a maximal ideal $\fom$, then we have a syzygy bundle $\Syz= \Syz(f_1 \comdots f_n)$ defined on $D(\fom)$.

\section{Forcing algebras and torsors}
\label{torsor}

Let $Z \subset \Spec R$ be the support of $M/N$ and let $U=\Spec R
-Z$ be the open complement where $\repmatrix$ is surjective. Let $s
\in M$ with forcing algebra $\forca$. We set $\torsor = \Spec \forca|_U$
and we assume that the fibers are non-empty (in the ideal case this
means that $f$ is not a unit). Then the action of the group scheme
$V$ on $\Spec \forca$ restricts to an action of the syzygy bundle
$\Syz=V|_U$ on $\torsor$, and this action is \emph{simply
transitive}. This means that locally the actions looks like the
action of $\Syz$ on itself by addition.

In general, if a vector bundle $\shS$ on a separated scheme $U$ acts
simply transitively on a scheme $\torsor \rightarrow U$ -- such a
scheme is called a geometric $\shS$-\emph{torsor} or an
\emph{affine-linear bundle} --, then this corresponds to a
cohomology class $c \in H^1(U, \shS)$ (where $\shS$ is now also the
sheaf of sections in the vector bundle $\shS$). This follows from
the $\rm \check{C}$ech cohomology by taking an open covering where
the action can be trivialized. Hence forcing data define, by
restricting the forcing algebra, a torsor
$\torsor$ over $U$.

On the other hand, the forcing data define the short exact sequence
$0 \ra \Syz \ra \O_U^n \ra \O_U^m \ra 0$ and $s$ is represented by
an element $s' \in R^m \rightarrow \Gamma(U, \O_U^m)$. By the
connecting homomorphism $s'$ defines a cohomology class
$$c =\delta (s') \in H^1(U, \Syz) \, .$$
An explicit computation of $\rm \check{C}$ech cohomology shows that
this class corresponds to the torsor given by the forcing algebra.

Starting from a cohomology class $c \in H^1(U,\shS)$, one may
construct a geometric model for the torsor classified by $c$ in the
following way: because of $H^1(U,\shS) \cong
\operatorname{Ext}^1(\O_U,\shS)$ we have an extension
$$0 \lra \shS \lra \shS' \lra \O_U \lra 0 \, .$$
This sequence induces projective bundles $\PP(\shS^\dual)
\hookrightarrow \PP(\shS'^\dual)$ and $T(c) \cong \PP(\shS'^\dual) -
\PP(\shS^\dual)$. If $\shS=\Syz(f_1 \comdots f_n)$ is the syzygy bundle for ideal generators, then the extension given by the cohomology class $\delta(f)$ coming from another element $f$ is easy to describe: it is just
$$0 \lra \Syz(f_1 \comdots f_n) \lra \Syz(f_1 \comdots f_n,f) \lra \O_U \lra 0 \, $$
with the natural embedding (extending a syzygy by zero in the last
component). This follows again from an explicit computation in $\rm
\check{C}$ech cohomology.

If the base $U$ is projective, a situation in which we will work
starting with Section \ref{gradedsection}, then $\PP(\shS'^\dual)$
is also a projective variety and $\PP(\shS^\dual)$ is a subvariety
of codimension one, a divisor. Then properties of the torsor are
reflected by properties of the divisor and vice versa.

\section{Forcing algebras and closure operations}
\label{closure}

A \emph{closure operation} for ideals or for submodules is an assignment
$$N \longmapsto N^c$$ for submodules $N \subseteq M$ of $R$-modules
$M$ such that $N \subseteq N^c =(N^c)^c$ holds. One often requires
further nice properties of a closure operation, like \emph{monotony}
or the \emph{independence of representation} (meaning that $s \in N^c$ if
and only
if $\bar{s} \in 0^c$ in $M/N$).
Forcing algebras are very natural objects to study such closure
operations.
The underlying philosophy is that $s \in N^c$ holds if and only if
the forcing morphism $\Spec \forca \rightarrow \Spec R$ fulfills a
certain property (depending on and characterizing the closure
operation).
The property is in general not uniquely determined; for the
identical closure operation one can take the properties to be
faithfully flat, to be
(cyclic) pure, or to have a
(module- or ring-) section.

Let us consider some easy closure operations to get a feeling for
this philosophy. In Section \ref{sectiontightsolid} we will see how
tight closure can be characterized with forcing algebras.

\begin{example}
For the \emph{radical} $\operatorname{rad} (I)$ the corresponding property is that $\Spec \forca \rightarrow \Spec
R$ is surjective. It is not surprising that a rough
closure operation corresponds to a rough property of a morphism. An
immediate consequence of this viewpoint is that we get at once a
hint of what the radical of a submodule should be: namely $s \in
\sqrt{N}
$
if and only if the forcing algebra is $\Spec$-surjective. This is
equivalent to the property that $s \otimes 1 \in \im (N \otimes_RK
\rightarrow M \otimes_RK)$ for all homomorphism $R \rightarrow K$ to
fields (or just for all $\kappa(\fop)$, $ \fop \in \Spec R$).
\end{example}

\begin{example}
We now look at the \emph{integral closure} of an ideal, which is
defined by
$$\bar{I}
= \{f \in R:\, \text{ there exists }
f^k +a_1f^{k-1}+ \ldots +a_{k-1}f + a_k=0,\, a_i \in I^{i}\} \, .$$
The integral closure was first used by Zariski as it describes
the normalization of blow-ups.
What is the corresponding property of a morphism?

We look at an example. For $R=K[X,Y]$ we have $X^2Y \in
\overline{(X^3,Y^3)}$ and $XY \not\in
\overline{(X^3,Y^3)}$. The inclusion follows from $(X^2Y)^3=X^6Y^3
\in (X^3,Y^3)^3$. The non-inclusion follows from the \emph{valuation
criterion} for integral closure: This says for a noetherian domain
$R$ that
$f
\in
\bar{I}$
if and only if for all mappings to discrete valuation rings $\varphi: R \rightarrow
V$ we have $\varphi(f) \in IV$. In the example the mapping $K[X,Y]
\rightarrow K[X]$, $ Y \mapsto X$, yields $X^2 \not\in (X^3)$, so it
can not belong to the integral closure. In both cases the mapping
$\Spec \forca \rightarrow \Spec R$ is surjective. In the second
case, the forcing algebra over the line $V(Y-X)$ is given by the
equation $T_1X^3+T_2X^3+X^2=X^2((T_1+T_2)X+1)$. The fiber over the
zero point is a plane and is an affine line over a hyperbola for
every other point of the line. The topologies above and below are
not much related:
The inverse image of the non-closed punctured line is closed, hence
the topology downstairs does not carry the \emph{image topology}
from upstairs. In fact, the relationship in general is
$$f \in \bar{I} \text{ if and only if } \Spec \forca \rightarrow
\Spec R
\text{ is universally a \emph{submersion} }$$ (a submersion in the
topological sense). This relies on the fact that both properties can
be checked with (in the noetherian case discrete) valuations
(for this criterion for submersions, see \cite{SGA1} and \cite{blicklebrennerintegral}).

Let us consider the forcing algebras for $(X,Y)$ and $1$ and for
$(X^2,Y^2)$ and $XY$ in $K[X,Y]$. The restricted spectra of the forcing algebras over the
punctured plane for these two forcing data are isomorphic, because
both represent the torsor given by the cohomology class
$\frac{1}{XY} = \frac{XY}{X^2Y^2} \in H^1(D(X,Y), \O)$. However, $XY
\in \overline{(X^2,Y^2)}$, but $1 \not\in \overline{(X,Y)}$ (not even in the
radical). Hence integral closure can be characterized by the forcing
algebra, but not by the induced torsor. An interesting feature of
tight closure is that it only depends on the cohomology class in the
syzygy bundle and the torsor induced by the forcing algebra
respectively.
\end{example}

\begin{example}
In the case of finitely generated algebras over the complex numbers
there is another interesting closure operation, called
\emph{continuous closure}. An element $s$ belongs to the continuous
closure of $N$ if the forcing algebra $\forca$ is such that the
morphism $\CC-\Spec
\forca \rightarrow \CC-\Spec R$ has a continuous section in the
\emph{complex topology}. For an ideal $I=(f_1 \comdots f_n)$ this is
equivalent to the existence of complex-continuous functions $g_1
\comdots g_n:\CC-\Spec R \rightarrow \CC$ such that $\sum_{i=1}^n
g_if_i=f$ (as an identity on $\CC-\Spec R$).
\end{example}

\begin{remark}
One can go one step further with the understanding of closure operations in terms of
forcing algebras. For this we take the forcing algebras which are
allowed by the closure operation (i.e., forcing algebras for
$s,N,M$, $s \in N^c$) and declare them to be the defining covers of a
(non-flat) \emph{Grothendieck topology}. This works basically for
all closure operations fulfilling certain natural conditions. This
embeds closure operations into the much broader context of
Grothendieck topologies, see \cite{brennergrothendieck}.
\end{remark}

\section{Tight closure as solid closure}
\label{sectiontightsolid}

We come back to tight closure, and its interpretation in terms of forcing algebras and
\emph{solid closure}.

\begin{theorem}
\label{solidtight}
Let $(R, \fom)$ be a local excellent normal domain of positive characteristic and let
$I$ denote an $\fom$-primary ideal. Then $f \in I^*$ if and only if
$H^{\dim R}_\fom (\forca) \neq 0$, where $A$ denotes the forcing algebra.
\end{theorem}
\begin{proof}
We indicate the proof of the direction that the cohomological property implies the
tight closure inclusion.
By the assumptions we may assume that $R$ is complete. Because of
$H^{\dim R}_{\fom} (\forca) \neq 0$ there exists by Matlis-duality a non-trivial $R$-module
homomorphism $\psi: \forca \rightarrow R$ and we may assume
$\psi(1)=:c
\neq 0$. In $\forca$ we have the equality $f = \sum_{i=1}^n f_iT_i$ and
hence
$$f^q =  \sum_{i=1}^n f_i^qT_i^q \text{ for all } q=p^{e} \, . $$
Applying the $R$-homomorphism $\psi$ to these equations gives
$$cf^q = \sum _{i=1}^n f_i^q \psi(T_i^q) \, ,$$
which is exactly the tight closure condition (the $ \psi(T_i^q)$ are
the coefficients in $R$). For the other direction see
\cite{hochstersolid}.
\end{proof}

This theorem provides the bridge between tight closure and cohomological properties of
forcing algebras. The first observation is that the property about
\emph{local cohomology}
on the right hand side does not refer to positive characteristic.
The closure operation defined by this property is called \emph{solid
closure}, and the theorem says that in positive characteristic and
under the given further assumptions
solid closure and tight closure coincide. The hope was that this
would provide a closure operation in all (even mixed)
characteristics with similar properties as tight closure. This hope
was however destroyed by the following example of Paul Roberts (see
\cite{robertscomputation})

\begin{example} (Roberts)
Let $K$ be a field of characteristic zero and consider
$$\forca=K[X,Y,Z]/(X^3T_1+Y^3T_2+Z^3T_3-X^2Y^2Z^2) \, .$$
Then
$H^3_{(X,Y,Z)}(\forca) \neq 0$. Therefore $X^2Y^2Z^2 \in (X^3,Y^3,Z^3)^{\solidclosure}$
in the regular ring $K[X,Y,Z]$. This means that in a
three-dimensional regular ring an ideal needs not be solidly-closed.
It is however an important property of tight closure that every
regular ring is \emph{$F$-regular}, namely that every ideal is
tightly closed. Hence solid closure is not a good replacement for
tight closure (for a variant called \emph{parasolid closure} with
all good properties in equal characteristic zero, see
\cite{brennerparasolid}).
\end{example}

Despite this drawback, solid closure provides an important interpretation of tight closure.
First of all we have for $d= \dim (R) \geq 2$ (the one-dimensional case is trivial) the
identities
$$H^d_\fom (\forca) \cong H^d_{\fom \forca} (\forca) \cong H^{d-1} (D(\fom \forca), \O) \, .$$
This easy observation is quite important. The open subset
$D(\fom \forca) \subseteq \Spec \forca$
is exactly the torsor induced by the forcing algebra over the
punctured spectrum $D(\fom) \subset \Spec R$. Hence we derive at an
important particularity of tight closure: tight closure of primary
ideals in  a normal excellent local domain depends only on the
torsor (or, what is the same, only on the cohomology class of the
syzygy bundle).
We recall from the last section that this property does not hold for integral closure.

By Theorem \ref{solidtight}, tight closure can be understood by studying the global sheaf
cohomology of the torsor given by a first cohomology class of the
syzygy bundle. The forcing algebra provides a geometric model for
this torsor. An element $f$ belongs to the tight closure if and only
if the \emph{cohomological dimension} of the torsor $T$ is $d-1$
(which is the cohomological dimension of
$D(\fom)$), and $f \not \in I^*$ if and only if the cohomological
dimension drops. Recall that the cohomological dimension of a scheme
$U$ is the largest number $i$ such that
$H^{i}(U, {\mathcal F}) \neq 0$ for some (quasi-)coherent sheaf $\mathcal F$ on $U$.
In the quasiaffine case, where $U \subseteq \Spec B$ (as in the case of torsors inside the
spectrum of the forcing algebra), one only has to look at the
structure sheaf ${\mathcal F}= \O$.

In dimension two this means that $f \in I^*$ if and only if the cohomological dimension of the
torsor is one, and $f \not\in I^*$ if and only if this is zero.
By a theorem of Serre (\cite[Theorem III.3.7]{hartshornealgebraic})
cohomological dimension zero means that $U$ is an \emph{affine
scheme}, i.e., isomorphic as a scheme to the spectrum of a ring (do
not confuse the ``affine'' in affine scheme with the ``affine'' in
affine-linear bundle).

It is in general a difficult question to decide whether a quasiaffine scheme is an affine scheme. Even in the special case of torsors there is no general machinery to
answer it. A necessary condition is that the complement has pure
codimension one (which is fulfilled in the case of torsors). So far
we have not gained any criterion from our geometric
interpretation.

\newcommand{\curv}{{C}}

\section{Tight closure in graded dimension two}
\label{gradedsection}

From now on we deal with the following situation:
$R$ is a two-dimensional normal standard-graded domain over an algebraically closed field of
any characteristic, $I=(f_1 \comdots f_n)$ is a homogeneous
$R_+$-primary ideal with homogeneous generators of degree $d_i=
\deg(f_i)$. Let $\curv = \Proj R$ be the corresponding smooth projective
curve. The ideal generators define the homogeneous resolution
$$0 \lra \Syz(f_1 \comdots f_n) \lra \oplus_{i=1}^n R(-d_i) \stackrel{f_1 \comdots f_n}{\lra} R \lra R/I \lra 0 \, ,$$
and the short exact sequence of vector bundles on $\curv$
$$0 \lra \Syz( f_1 \comdots f_n) \lra \oplus_{i=1}^n \O_\curv(-d_i)
 \stackrel{f_1 \comdots f_n}{\lra} \O_\curv \lra  0 \, .$$
We also need the $m$-twists of this sequence for every $m \in \ZZ$,
$$0 \lra \Syz( f_1 \comdots f_n)(m) \lra \oplus_{i=1}^n \O_\curv(m-d_i)
 \stackrel{f_1 \comdots f_n}{\lra} \O_\curv(m) \lra  0 \, .$$
It follows from this \emph{presenting sequence} by the additivity of rank and degree
that the vector bundle $\Syz(f_1 \comdots f_n)(m)$ has rank $n-1$
and degree $$(m(n-1) - \sum_{i=1}^n d_i) \deg \curv$$
(where $\deg \curv = \deg \O_\curv(1)$ is the degree of the curve).

A homogeneous element $f \in R_m= \Gamma(\curv, \O_\curv(m))$
defines again a cohomology class $c \in H^1(\curv, \O_\curv(m))$ as
well as a torsor $T(c) \rightarrow \curv$. This torsor is a
homogeneous version of the torsor induced by the forcing algebra on
$D(m) \subset \Spec R$. This can be made more precise by endowing
the forcing algebra $\forca=R[T_1 \comdots T_n]/(f_1T_1 \plusdots
f_nT_n-f)$ with a (not necessarily positive) $\ZZ$-grading and
taking $T=D_+(R_+) \subseteq \Proj A$. From this it follows that the
affineness of this torsor on $\curv$ is decisive for tight closure.
The translation of the tight closure problem via forcing algebras
into torsors over projective curves has the following advantages:

\begin{enumerate}

\item
We can work over a smooth projective curve, i.e., we have reduced the dimension of the base and we have removed the singularity.

\item
We can work in a projective setting and use intersection theory.

\item
We can use the theory of vector bundles, in particular the notion of semistable bundles and their moduli spaces.
\end{enumerate}

We will give a criterion when such a torsor is affine and hence when
a homogeneous element belongs to the tight closure of a graded
$R_+$-primary ideal. For this we need the following definition.

\begin{definition}
Let $\shS$ be a locally free sheaf on a smooth projective curve
$\curv$. Then $\shS$ is called \emph{semistable}, if $\deg(\shT)/\rk
(\shT) \leq \deg(\shS)/\rk (\shS) $ holds for all subbundles $\shT \neq 0$. In positive characteristic,
$\shS$ is called \emph{strongly semistable}, if all \emph{Frobenius pull-backs} $F^{e*}( \shS)$, $e \geq 0$, are semistable (here $F: C
\rightarrow C$ denotes the \emph{absolute Frobenius} morphism).
\end{definition}

Note that for the syzygy bundle we have the natural isomorphism (by pulling back the presenting sequence) $$F^{e*}(\Syz(f_1 \comdots f_n )) \cong \Syz(f_1^q \comdots f_n^q) \, .$$
Therefore the Frobenius pull-back of the cohomology class $\delta(f) \in $ \linebreak $H^1(C,\Syz(f_1 \comdots f_n )(m))$ is
$$ F^{e*}( \delta(f)) = \delta(f^q) \in H^1(C,\Syz(f_1^q \comdots f_n^q )(qm)) \, .$$

The following two results establish an exact numerical \emph{degree bound} for tight closure under the condition that the syzygy bundle is strongly semistable.

\begin{theorem}
\label{ssinclusion}
Suppose that $\Syz(f_1 \comdots f_n)$ is strongly semistable. Then we have $R_m \subseteq I^*$ for $m \geq (\sum_{i=1}^n d_i)/(n-1)$.
\end{theorem}
\begin{proof}
Note that the degree condition implies that $\shS:=\Syz(f_1 \comdots f_n)(m)$ has non-negative degree. Let $c \in H^1(\curv,\shS)$ be any cohomology class (it might be $\delta(f)$ for some $f$ of degree $m$). The pull-back $F^{e*}(c)$ lives in $H^1(\curv,F^{e*}(\shS))$. Let now $k$ be such that $\O_\curv(-k) \otimes \omega_\curv $ has negative degree, where $\omega_\curv$ is the canonical sheaf on the curve. Let $z \in \Gamma(\curv, \O_\curv (k))=R_k$, $z \neq 0$. Then $zF^{e*}(c) \in H^1(\curv, F^{e*}(\shS ) \otimes \O_\curv(k))$. However, by degree considerations, these cohomology groups are zero: by Serre duality they are dual to
$H^0( \curv, F^{e*}( \shS ^{\dual} ) \otimes \O_\curv (-k) \otimes \omega_\curv )$,
and this bundle is semistable of negative degree, hence it can not
have global sections. Because of $z F^{e*}(c)=0$ it follows that $z
f^q$ is in the image of the mapping given by $f_1^q \comdots f_n^q$,
so $z f^q \in I^{[q]}$ and $f \in I^*$.
\end{proof}

\begin{theorem}
\label{ssexclusion}
Suppose that $\Syz(f_1 \comdots f_n)$ is strongly semistable. Let $m < (\sum_{i=1}^n d_i)/(n-1)$ and let $f$ be a homogeneous element of degree $m$. Suppose that $f^{{p^a}} \not\in I^{[p^a]}$ for $a$ such that $p^a > gn(n-1)$ (where $g$ is the genus of $C$). Then $f \not\in I^*$.
\end{theorem}
\begin{proof}
Here the proof works with the torsor $T$ defined by $c=\delta(f)$.
The syzygy bundle $\shS=\Syz(f_1\comdots f_n)(m)$ has now negative degree, hence its
dual bundle $\shF=\shS^\dual$ is an \emph{ample} vector bundle (as
it is strongly semistable of positive degree). The class defines a
non-trivial dual extension $0 \ra \O_\curv \ra \shF' \ra \shF \ra
0$. By the assumption also a certain Frobenius pull-back of this
extension is still non-trivial. Hence $\shF'$ is also ample and
therefore $\PP(\shF) \subset \PP(\shF')$ is an ample divisor and its
complement $T=\PP(\shF') - \PP(\shF)$ is affine. Hence $f \not\in
I^*$.
\end{proof}

It is in general not easy to establish whether a bundle is strongly semistable or not.
However, even if we do not know whether the syzygy bundle is strongly semistable,
we can work with its \emph{strong Harder-Narasimhan filtration}. The Harder Narasimhan filtration of a vector bundle $\shS$ on a smooth projective curve is a filtration
$$ 0 = \shS_0 \subset \shS_1 \subset \shS_2 \subset \ldots  \subset \shS_{t-1} \subset \shS_{t} = \shS $$
with $\shS_i/\shS_{i-1}$ semistable and descending slopes $$\mu
(\shS_1) > \mu (\shS_2/\shS_1) > \ldots > \mu (\shS/\shS_{t-1}) \, .$$
Since the Frobenius pull-back of a semistable bundle need not be
semistable anymore, the Harder-Narasimhan filtration of $F^*(\shS)$
is quite unrelated to the Harder-Narasimhan filtration of $\shS$.
However, by a result of A. Langer \cite[Theorem 2.7]{langersemistable}, there exists
a certain number $e$ such that the quotients in the
Harder-Narasimhan filtration of $F^{e*}(\shS)$ are strongly semistable.
Such a filtration is called \emph{strong}. With a strong Harder-Narasimhan
filtration one can now formulate an exact numerical criterion for
tight closure inclusion building on Theorems \ref{ssinclusion} and
\ref{ssexclusion}.

The criterion basically says that a torsor is affine (equivalently, $f \not \in I^*$),
if and only if the cohomology class is non-zero in some strongly semistable quotient
of negative degree of the strong Harder-Narasimhan filtration. One
should remark here that even if we start with a syzygy bundle, the
bundles in the filtration are no syzygy bundles, hence it is
important to develop the theory of torsors of vector bundles in full
generality.
From this numerical criterion one can deduce an answer to the
tantalizing question.

\begin{theorem}
\label{tantalizingtheorem}
Let $K= \overline{\FF_p}$ be the algebraic closure of a finite field
and let $R$ be a normal standard-graded $K$-algebra of dimension
two. Then $I^*=I^+$ for every $R_+$-primary homogeneous ideal.
\end{theorem}
\begin{proof}
This follows from the numerical criterion for the affineness of torsors mentioned above.
The point is that the same criterion holds for the non-existence of
projective curves inside the torsor. One reduces to the situation of
a strongly semistable bundle $\shS$ of degree $0$. Every cohomology
class of such a bundle defines a non-affine torsor and hence we have to show that
there exists a projective curve inside, or equivalently, that the
cohomology class can be annihilated by a finite cover of the curve.
Here is
where the finiteness assumption about the field enters. $\shS$ is
defined over a finite subfield $\FF_q \subseteq K$, and it is
represented (or rather, its $S$-equivalence class) by a point in the moduli space of semistable bundles of
that rank and degree $0$. The Frobenius pull-backs $F^{e*}(\shS)$
are again semistable (by strong semistability) and they are defined
over the same finite field.
Because semistable bundles form a bounded family
(itself the reason for the existence of the moduli space), there exist only finitely
many semistable bundles defined over
$\FF_q$ of degree zero.
Hence there exists a repetition, i.e., there exists $e' > e$ such
that we have an isomorphism $F^{e'*}(\shS) \cong F^{e*}(\shS)$. By a
result of H. Lange and U. Stuhler \cite{langestuhler} there exists a
finite mapping
$\curv' \stackrel{\varphi}{\ra} \curv \stackrel{F^{e}}{\ra} \curv$ (with $\varphi$ \'{e}tale)
such that the pull-back of the bundle is trivial. Then one is left
with the problem of annihilating a cohomology class $c
\in H^1(\curv, \O_\curv)$, which is possible using Artin-Schreier
theory (or the graded version of K. Smith's parameter theorem,
\cite{smithparameter}).
\end{proof}

\begin{remark}
This theorem was extended by G. Dietz for $R_+$-primary ideals which are not
necessarily homogeneous \cite{dietztightclosure}. The above proof shows
how important the assumption is that the base field is finite or the
algebraic closure of a finite field. Indeed, we will see in the last
section that the statement is not true when the base field contains
transcendental elements. Also some results on Hilbert-Kunz functions
depend on the property that the base field is finite.
\end{remark}

\newcommand{\shH}{{\mathcal H}}

\section{Applications to Hilbert-Kunz theory}
\label{hilbertkunz}

The geometric approach to tight closure was also successful in
Hilbert-Kunz theory. This theory originates in the work of E. Kunz
(\cite{kunzcharacterization}, \cite{kunznoetherian}) and was largely extended by P. Monsky
(\cite{monskyhilbertkunz}, \cite{hanmonsky}).

Let $R$ be a commutative ring of positive characteristic and let $I$
be an ideal which is primary to a maximal ideal. Then all
$R/I^{[q]}$, $q=p^{e}$, have finite length, and the \emph{Hilbert-Kunz
function} of the ideal is defined to be
$$ e \longmapsto \varphi(e)= \lg (R/ I^{[p^{e}]})\, . $$

Monsky proved the following fundamental theorem of Hilbert-Kunz theory
(\cite{monskyhilbertkunz}, \cite[Theorem 6.7]{hunekeapplication}).

\begin{theorem}
The limit
$$ \lim_{e \mapsto \infty} \frac{\varphi(e)}{p^{e \dim (R) }} $$
exists (as a positive real number) and is called the
\emph{Hilbert-Kunz multiplicity} of $I$, denoted by $e_{\rm HK}
(I)$.
\end{theorem}

The Hilbert-Kunz multiplicity of the maximal ideal in a local ring
is usually denoted by $e_{\rm HK}(R)$ and is called the Hilbert-Kunz
multiplicity of $R$.
It is an open question whether this number is always rational.
Strong numerical evidence suggests that this is probably not true in
dimension $\geq 4$, see \cite{monskylikelycounterexample}.
We will deal with the two-dimensional situation in a minute, but
first we relate Hilbert-Kunz theory to tight closure (see
\cite[Theorem 5.4]{hunekeapplication}).

\begin{theorem}
\label{hilbertkunztightclosure}
Let $R$ be an analytically unramified and formally equidimensional
local ring of positive characteristic and let $I$ be an
$\fom$-primary ideal. Let $f \in R$. Then $f \in I^*$ if and only if
$$e_{\rm HK} (I) = e_{\rm HK} ((I,f)) \, .$$
\end{theorem}

This theorem means that the Hilbert-Kunz multiplicity is related to
tight closure in the same way as the Hilbert-Samuel multiplicity is
related to integral closure.

We restrict now again to the case of an $R_+$-primary homogeneous ideal in
a standard-graded normal domain $R$ of dimension two over an
algebraically closed field $K$ of positive characteristic $p$. In
this situation Hilbert-Kunz theory is directly related to global
sections of the Frobenius pull-backs of the syzygy bundle on $\Proj
R$ (see
Section \ref{gradedsection}). We shall see that it is possible to
express the Hilbert-Kunz multiplicity in terms of the strong
Harder-Narasimhan filtration of this bundle.

For homogeneous ideal generators $f_1 \comdots f_n$ of degrees $d_1,
\ldots , d_n$  we write down again the presenting sequence on $C= \Proj
R$,
$$0 \lra \Syz(f_1\comdots f_n) \lra \bigoplus_{i=1}^n \O_C(-d_i)
\stackrel{f_1\comdots f_n}{\lra } \O_C \lra 0 \, . $$
The $m$-twists of the Frobenius pull-backs of this sequence are
$$0 \lra \Syz(f_1^q\comdots f_n^q)(m) \lra \bigoplus_{i=1}^n \O_C(m-qd_i)
 \stackrel{f^q_1\comdots f^q_n}{\lra} \O_C(m) \lra 0 \, . $$
The global evaluation of the last short exact sequence is
$$0 \lra \Gamma(C,\Syz(f_1^q\comdots f_n^q)(m)) \lra \bigoplus_{i=1}^n
R_{m-qd_i}
\stackrel{f^q_1\comdots f^q_n}{\lra} R_m  \,  ,$$
and the cokernel of the map on the right is
$$R_m/(f_1^q \comdots f_n^q) = ( R/I^{[q]})_m \, . $$
Because of $R/I^{[q]} = \bigoplus_m (R/I^{[q]})_m$, the length of
$R/I^{[q]}$ is the sum over the degrees $m$ of the $K$-dimensions of
these cokernels. The sum is in fact finite  because the ideals
$I^{[q]}$ are primary (or because $H^1(C, \Syz(f_1^q \comdots
f_n^q)(m))=0$ for $m \gg 0$), but the bound for the summation grows
with $q$. Anyway, we have
\begin{eqnarray*}
\dim (R/I^{[q]})_m  &=&
\dim (\Gamma( C, \O_C(m))) - \sum_{i=1}^n \dim (\Gamma( C,
\O_C(m-qd_i))) \cr
& &+ \dim (\Gamma(C, \Syz(f_1^q \comdots f_n^q)(m))) \, .
\end{eqnarray*}
The computation of the dimensions $\dim (\Gamma( C, \O_C(\ell)))$ is
easy, hence the problem is to control the global sections of
$\Syz(f_1^q \comdots f_n^q)(m)$, more precisely, its behavior for
large $q$, and its sum over a suitable range of $m$. This behavior
is encoded in the strong Harder-Narasimhan filtration of the syzygy
bundle. Let $e$ be fixed and large enough such that the
Harder-Narasimhan filtration of the pull-back
$\shH = F^{e*}(\Syz(f_1 \comdots f_n) ) =\Syz(f_1^q \comdots f_n^q)$ is strong.
Let $\shH_j \subseteq \shH $, $j=1 \comdots t$, be the subsheaves
occurring in the Harder-Narasimhan filtration and set
$$\nu_j := \frac{- \mu ( \shH_j/\shH_{j-1})}{p^{e} \deg (C)} \, \text{ and } r_j
= \rk ( \shH_j/\shH_{j-1}) \, .$$
Because the Harder-Narasimhan filtration of $\shH$ and of all its pull-backs is
strong, these numbers are independent of $e$.
The following theorem was shown by Brenner and Trivedi independently
(\cite{brennerhilbertkunzrationality}, \cite{trivedihilbertkunz}).

\begin{theorem}
\label{hkformula}
Let $R$ be a normal two-dimensional standard-graded domain over an
algebraically closed field and let $I=(f_1 \comdots f_n)$ be a
homogeneous $R_+$-primary ideal, $d_i = \deg(f_i)$.
Then the Hilbert-Kunz multiplicity of $I$ is given by the formula
$$ e_{\rm HK} (I)
= \frac{\deg (C)}{2} \left( \sum_{j=1}^t r_j \nu_j^2 - \sum_{i=1}^n d_i^2 \right)\, .$$
In particular, it is a rational number.
\end{theorem}

We can also say something about the behavior of the Hilbert-Kunz
function under the additional condition that the base field is the
algebraic closure of a finite field (see
\cite{brennerhilbertkunzfunction}).

\begin{theorem}
\label{hilbertkunzperiodic}
Let $R$ and $I$ be as before and suppose that the base field is the
algebraic closure of a finite field. Then the Hilbert-Kunz function
has the form
$$ \varphi(e)= e_{\rm HK}(I) p^{2e} + \gamma(e) \, ,$$
where $\gamma$ is eventually periodic.
\end{theorem}

This theorem also shows that here the ``linear term'' in the
Hilbert-Kunz function exists and that it is zero. It was proved in
\cite{hunekemcdermottmonsky} that for normal excellent $R$ the
Hilbert-Kunz function looks like
$$e_{HK} q^{\dim (R)} + \beta q^{\dim
(R)-1} + \text{smaller terms} \, .$$
For possible behavior of the second term in the non-normal case in dimension two see
\cite{monskyirreducible}. See also Remark \ref{hilbertkunzarithmetic}.

\section{Arithmetic and geometric deformations of tight closure}
\label{deformation}

The geometric interpretation of tight closure theory led to a fairly
detailed understanding of tight closure in graded dimension two. The
next easiest case is to study how tight closure behaves in families
of two-dimensional rings, parametrized by a one dimensional ring.
Depending on whether the base has mixed characteristic (like
$\Spec \ZZ$) or equal positive characteristic $p$ (like $\Spec K[T]=\mathbb A^1_K$)
we talk about \emph{arithmetic} or \emph{geometric deformations}.

More precisely, let $D$ be a one-dimensional domain and let $S$ be a $D$-standard-graded domain
of dimension three, such that for every point $\fop \in \Spec D$ the
fiber rings $S_{\kappa(\fop)} = S \otimes_D \kappa(\fop)$ are normal
standard-graded domains over $\kappa(\fop)$ of dimension two. The
data
$I=(f_1 \comdots f_n)$ in
$S$ and $f \in S$ determine corresponding data in these fiber rings, and the syzygy bundle
$\Syz(f_1 \comdots f_n)$ on $\Proj S \rightarrow \Spec D$ determines syzygy bundles on the curves
$ C_{\kappa(\fop)} = \Proj S_{\kappa(\fop)} $.
The natural questions here are: how does the property $f \in I^*$ (in $S_{\kappa(\fop)}$) depend on $\fop$, how does $e_{\rm HK}(I)$ depend on $\fop$, how does
strong semistability depend on $\fop$, how does the affineness of
torsors depend on $\fop$?

Semistability itself is an open property and behaves nicely in a family in the sense that
if the syzygy bundle is semistable on the curve over the generic point, then it is semistable
over almost all closed points. D. Gieseker gave in \cite{giesekerfrobenius} an example of a collection of bundles
such that, depending on the parameter, the $e$th pull-back is
semistable, but the $(e+1)$th is not semistable anymore (for every
$e$). The problem how strong semistability behaves under arithmetic deformations was explicitly
formulated by Y. Miyaoka and by N. Shepherd-Barron (\cite{miyaokachern},
\cite{shepherdbarronsemistability}).

In the context of Hilbert-Kunz theory, this question has been
studied by P. Monsky (\cite{hanmonsky}, \cite{monskyfamilyquartic}), both in the
arithmetic and in the geometric case.
Monsky (and Han) gave examples that the Hilbert-Kunz multiplicity may vary in a family.

\begin{example}
Let $R_p= \ZZ/(p)[X,Y,Z]/(X^4+Y^4+Z^4)$. Then the Hilbert-Kunz multiplicity of the maximal ideal
is
$$e_{\rm HK}(R_p)= \begin{cases} 3 \text{ for } p
= \pm 1 \mod 8 \cr 3 + 1/p^2 \text{ for } p = \pm 3 \mod 8 \end{cases} \, .$$
Note that by the theorem on prime numbers in arithmetic progressions there exist
infinitely many prime numbers of all these congruence types.
\end{example}

\begin{remark}
\label{hilbertkunzarithmetic}
In the previous example there occur infinitely many different values
for $e_{\rm HK}(R_p)$ depending on the characteristic, the limit as
$p \mapsto \infty$ is however well defined. Trivedi showed
\cite{trivedihilbertkunzreduction} that in the graded two-dimensional situation this
limit always exists, and that this limit can be computed by the
Harder-Narasimhan filtration of the syzygy bundle in characteristic
zero. Brenner showed that one can define, using this
Harder-Narasimhan filtration, a Hilbert-Kunz multiplicity directly
in characteristic zero, and that this Hilbert-Kunz multiplicity
characterizes solid closure \cite{brennerhilbertkunzcriterion} in the
same way as Hilbert-Kunz multiplicity characterizes tight closure
in positive characteristic (Theorem \ref{hilbertkunztightclosure}
above). Combining these results one can say that ``solid closure is
the limit of tight closure'' in graded dimension two, in the sense
that $f \in I^{\solidclosure}$ in characteristic zero if and only if the
Hilbert-Kunz difference $e_{\rm HK} ((I,f)) - e_{\rm HK} (I)$ tends
to $0$ for $p \mapsto \infty $.

It is an open question whether in all dimensions the Hilbert-Kunz
multiplicity has always a limit as $p$ goes to infinity, whether
this limit, if it exists, has an interpretation in characteristic
zero alone (independent of reduction to positive characteristic) and
what closure operation it would correspond to.
See also \cite{brennerlimiller}.
\end{remark}

In the geometric case, Monsky gave the following example (\cite{monskyfamilyquartic}).

\begin{example}
\label{monskyquartic}
Let $K=\ZZ/(2)$ and let
$$S= \ZZ/(2)[T][X,Y,Z]/(Z^4+Z^2XY+Z(X^3+Y^3)+(T+T^2)X^2Y^2 ) \, .$$
We consider $S$ as an algebra over $\ZZ/(2)[T]$ ($T$ has degree $0$).
Then the Hilbert-Kunz multiplicity of the maximal ideal is
$$e_{\rm HK}\! (S_{\kappa (\fop)})
\! = \! \begin{cases} 3 \text{ if } \kappa (\fop) =K(T) \text{ (generic case)}\cr
3 + 1/4^m \text{ if } \kappa(\fop)= \ZZ/(2)(t) \text{ is finite over
}
\ZZ/(2)
\text{ of degree } m . \end{cases} $$
\end{example}

By the work of
Brenner and Trivedi (see Section \ref{hilbertkunz})
these examples can be translated immediately into examples where
strong semistability behaves weirdly. From the first example we get
an example of a vector bundle of rank two over a projective relative
curve over $\Spec \ZZ$ such that the bundle is semistable on the
generic curve (in characteristic zero), and is strongly semistable
for infinitely many prime reductions, but also not strongly
semi\-stable for infinitely many prime reductions.

From the second example we get an example of a vector bundle of rank two over a projective
relative curve over the affine line $\mathbb A^1_{\ZZ/(2)}$, such
that the bundle is strongly semistable on the generic curve (over
the function field $
\ZZ/(2) (T)$), but not strongly semistable for the curve over any
finite field (and the degree of the field extension determines which
Frobenius pull-back destabilizes).

To get examples where tight closure varies with the base one has to go one step further
(in short, weird behavior of Hilbert-Kunz multiplicity is a necessary condition for weird behavior of tight closure). Interesting
behavior can only happen for elements of degree $(\sum d_i)/(n-1)$
(the degree bound, see Theorems \ref{ssinclusion} and \ref{ssexclusion}).

In \cite{brennerkatzmanarithmetic}, Brenner and M. Katzman showed
that tight closure does not behave uniformly under an arithmetic
deformation, thus answering negatively a question in
\cite{hochstersolid}.

\begin{example}
\label{brennerkatzmanexample}
Let $$R=\ZZ/(p)[X,Y,Z]/(X^7+Y^7+Z^7)$$ and $I=(X^4,Y^4,Z^4)$, $f=X^3Y^3$. Then
$f \in I^*$ for $p=3 \mod 7$ and $f \not \in I^*$ for $p=2 \mod 7$
(see \cite[Proposition 2.4 and Proposition 3.1]{brennerkatzmanarithmetic}). Hence we have infinitely many prime reductions
where the element belongs to the tight closure and infinitely many prime reductions
where it does not.
\end{example}

\begin{remark}
Arithmetic deformations are closely related to the question ``what is tight closure over a field of characteristic zero''. The general philosophy is that characteristic zero behavior of tight closure should reflect
the behavior of tight closure for almost all primes, after
expressing the relevant data over an arithmetic base. By declaring
$f \in I^*$, if this holds for almost all primes, one obtains a
satisfactory theory of tight closure in characteristic zero with the
same impact as in positive characteristic. This is a systematic way
to do reduction to positive characteristic (see
\cite[Appendix 1]{hunekeapplication} and \cite{hochsterhuneketightclosurezero}).
However, the example above shows that there is not always a uniform
behavior in positive characteristic. A consequence is also that
solid closure in characteristic zero is not the same as tight
closure (but see Remark \ref{hilbertkunzarithmetic}). From the
example we can deduce that
$f
\in I^{\solidclosure}$, but $f \not\in I^*$ in $\QQ[X,Y,Z]/(X^7+Y^7+Z^7)$. Hence,
the search for a good tight closure operation in characteristic zero
remains.
\end{remark}

We now look at geometric deformations. They are directly related to the localization problem and to the tantalizing problem which we have mentioned in the introduction.

\begin{lemma}
Let $D$ be a one-dimensional domain of finite type over $\ZZ/(p)$ and let $S$ be a
$D$-domain of finite type.
Let $f\in S$ and $I \subseteq S$ be an ideal. Suppose that localization holds for $S$. If then $f \in I^*$ in the generic fiber ring $S_{Q(D)}$, then also $f \in I^*$ in $S_{\kappa(\fop)} = S \otimes_D \kappa(\fop)$ for almost all closed points $\fop \in \Spec D$.
\end{lemma}
\begin{proof}
The generic fiber ring is the localization of $S$ at the multiplicative system $M =D -\{0\}$ (considered in $S$). So if $f \in I^*$  holds in $S_{Q(D)}=S_M$, and if localization holds, then there exists $h \in M $ such that $hf \in I^*$ in $S$ (the global ring of the deformation).
By the persistence of tight closure (\cite[Theorem 2.3]{hunekeapplication} applied to $S \rightarrow
S_{\kappa(\fop)}$)
it follows that $hf \in I^*$ in $S_{\kappa(\fop)}$ for all closed points $\fop \in \Spec D$.
But $h$ is a unit in almost all residue class fields $\kappa(\fop)$, so the result follows.
\end{proof}

\begin{example}
\label{brennermonskyexample}
Let $$S=\ZZ/(2)[T][X,Y,Z]/(Z^4+Z^2XY+Z(X^3+Y^3)+(T+T^2)X^2Y^2)$$ as in Example
\ref{monskyquartic} and let $I=(X^4,Y^4,Z^4)$, $f=Y^3Z^3$ ($X^3Y^3$ would not work).
Then $f \in I^*$, as is shown in \cite{brennermonsky}, in the generic fiber ring
$S_{\ZZ/(2)(T)}$, but $f \not\in I^*$ in $S_{\kappa(\fop)}$ for all closed points
$\fop \in \Spec D$.
Hence tight closure does not commute with localization.
\end{example}

\begin{example}
Let $K=\ZZ/(2)(T)$ and  $R=K[X,Y,Z]/(Z^4+Z^2XY+Z(X^3+Y^3)+(T+T^2)X^2Y^2)$.
This is the generic fiber ring of the previous example. It is a
normal, standard-graded domain of dimension two and it is defined
over the function field.
In this ring we have $Y^3Z^3 \in (X^4,Y^4,Z^4)^*$, but $Y^3Z^3 \not\in (X^4,Y^4,Z^4)^+ $. Hence tight closure is not the same as plus closure, not even in dimension two.
\end{example}

\section{Some open problems}
\label{problems}

We collect some open questions and problems, together with some comments of what
is known and some guesses. We first list problems which are weaker variants of the localization problem.

\begin{question}
Is $F$-regular the same as weakly $F$-regular?
\end{question}

Recall that a ring is called \emph{weakly $F$-regular} if every
ideal is tightly closed, and \emph{$F$-regular} if this is true for
all localizations. A positive answer would have followed from a
positive answer to the localization problem. This path is not
possible anymore, but there are many positive results on this: it is
true in the Gorenstein case, in the graded case (\cite{lyubezniksmithstrongweak}), it is true over an uncountable field
(proved by Murthy, see \cite[Theorem 12.2]{hunekeapplication}). All
this shows that a positive result is likely, at least under some
additional assumptions.

\begin{question}
Does tight closure commute with the localization at one element?
\end{question}

There is no evidence why this should be true. It would be nice to
see a counterexample, and it would also be nice to have examples of
bad behavior of tight closure under geometric deformations in all
characteristics.

\begin{question}
Suppose $R$ is of finite type over a finite field. Is tight closure
the same as plus closure?
\end{question}

This is known in graded dimension two for $R_+$-primary
ideals by Theorem \ref{tantalizingtheorem} and the extension for
non-homogeneous ideals (but still graded ring) by Dietz (see \cite{dietztightclosure}). To attack
this problem one probably needs first to establish new exact
criteria of what tight closure is. Even in dimension two, but not graded,
the best way to establish results is probably to develop a theory of strongly
semistable modules on a local ring.

Can one characterize the rings where tight closure is plus closure? Are rings, where every ideal coincides with its plus closure, $F$-regular?

\medskip
For a two-dimensional standard-graded domain and the corresponding projective curve, the following problems remain.

\begin{question}
Let $C$ be a smooth projective curve over a field of positive
characteristic, and let $\shL$ be an invertible sheaf of degree
zero. Let $c \in H^1(C, \shL)$ be a cohomology class. Does there
exist a finite mapping $C' \rightarrow C$, $C'$ another projective
curve, such that the pull-back annihilates $c$.
\end{question}

This is known for the structure sheaf $\O_C$ and holds in general over (the
algebraic closure of) a finite field. It is probably not true over a
field with transcendental elements, the heuristic being that
otherwise there would be a uniform way to annihilate the class over
every finite field (an analogue is that every invertible sheaf of degree zero over a
finite field has finite order in $\operatorname{Pic}^0(C)$, but the orders do not have much in
common as the field varies, and the order over larger fields might
be infinite).

\begin{question}
Let $R$ be a two-dimensional normal standard-graded domain and let $I$ be an
$R_+$-primary homogeneous ideal. Write $\varphi(e)= e_{HK} p^2 + \gamma(e)$. Is $\gamma(e)$ eventually periodic?
\end{question}

By Theorem \ref{hilbertkunzperiodic} this is true if the base field is finite, but this question is open if the base field contains transcendental elements. How does (the lowest term of) the  Hilbert-Kunz function behave under a geometric deformation?

\begin{question}
Let $C \rightarrow \Spec D$ be a relative projective curve over an arithmetic base like $\Spec \ZZ$, and let
$\shS$ be a vector bundle over $C$. Suppose that the generic bundle $\shS_0$ over the generic curve of characteristic zero is semistable. Is then $\shS_p$ over $C_p$ strongly semistable for infinitely many prime numbers $p$?
\end{question}

This question was first asked by Y. Miyaoka (\cite{miyaokachern}). Corresponding questions for an arithmetic family of two-dimensional rings are: Does there exist always infinitely many prime numbers where the Hilbert-Kunz multiplicity coincides with the characteristic zero limit? If an element belongs to the solid closure in characteristic zero, does it belong to the tight closure for infinitely many prime reductions? In \cite{brennermiyaoka}, there is a series of examples where the number of primes with not strongly semistable reduction has an arbitrary small density under the assumption that there exist infinitely many Sophie Germain prime numbers (a prime number $p$ such that also $2p+1$ is prime).

\medskip
We come back to arbitrary dimension.

\begin{question}
Understand tight closure geometrically, say for standard-graded
normal domains with an isolated singularity. The same for
Hilbert-Kunz theory.
\end{question}

Some progress in this direction has been made in \cite{brennerlinearfrobenius} and in
\cite{brennerfischbacherweitzgenericbounds}, but much more has to be done. What
is apparent from this work is that positivity properties of the
top-dimensional syzygy bundle coming from a resolution are
important. A problem is that strong semistability controls global
sections and by Serre duality also top-dimensional cohomology, but
one problem is to control the intermediate cohomology.

\begin{question}
Find a good closure operation in equal characteristic zero, with tight
closure like properties, with no reduction to positive
characteristic.
\end{question}

The notion of \emph{parasolid closure} gives a first answer to
this \cite{brennerparasolid}. However, not much is known about it
beside that it fulfills the basic properties one expects from tight
closure, and many proofs depend on positive characteristic (though
the notion itself does not). Is there a more workable notion?

One should definitely try to understand here several candidates with the help of
forcing algebras and the corresponding Grothendieck topologies. A
promising approach is to allow the forcing algebras as coverings
which do not annihilate (top-dimensional) local cohomology unless it
is annihilated by a resolution of singularities.

Is there a closure operation which commutes with localization
(this is also not known for characteristic zero tight closure, but probably false)?

\begin{question}
Find a good closure operation in mixed characteristic and prove the
remaining homological conjectures.
\end{question}

\medskip
In Hilbert-Kunz theory, the following questions are still open.

\begin{question}
Is the Hilbert-Kunz multiplicity always a rational number? Is it at least an
algebraic number?
\end{question}

The answer to the first question is probably no,
as the numerical material in \cite{monskylikelycounterexample} suggests. However, this
still has to be established.

\begin{question}
Prove or disprove that the Hilbert-Kunz multiplicity has always a
limit as the characteristic tends to $\infty$.
\end{question}

If it has, or in the cases where it has, one should also find a direct interpretation in characteristic zero and study the corresponding closure operation.

\bibliographystyle{amsplain}

\bibliography{bibliothek}

\end{document}